\documentclass[a4paper,12pt]{article}

\usepackage{bbm}
\usepackage{amsmath, wrapfig}
\usepackage{dsfont, a4wide, amsthm, amssymb, amsfonts, graphicx}
\usepackage{fancyhdr, xspace, psfrag, setspace, supertabular, color, enumerate}
\usepackage{hyperref}

\newtheorem{theorem}{Theorem}[section]
\newtheorem{proposition}[theorem]{Proposition}
\newtheorem{lemma}[theorem]{Lemma}

\newtheorem{corollary}[theorem]{Corollary}

\newtheorem{definition}[theorem]{Definition}

\newtheorem{example}[theorem]{Example}
\newtheorem{question}[theorem]{Question}

\theoremstyle{remark}

\def\a{\alpha}

\def\e{\epsilon}
\def\k{\kappa}

\def\cf{\mathcal{F}}

\def\F{\mathbb{F}}

\def\E{\mathbb{E}}
\def\P{\mathbb{P}}
\def\N{\mathbb {N}}

\DeclareMathOperator{\rk}{rk}

\title{Ranges of polynomials control degree ranks of Green and Tao over finite prime fields}

\author{Thomas Karam\footnote{Mathematical Institute, University of Oxford. Email: \texttt{thomas.karam@maths.ox.ac.uk}.}}

\begin{document}


\maketitle

\begin{abstract}

Let $p$ be a prime, let $1 \le t < d < p$ be integers, and let $S$ be a non-empty subset of $\F_p$. We establish that if a polynomial $P:\F_p^n \to \F_p$ with degree $d$ is such that the image $P(S^n)$ does not contain the full image $A(\F_p)$ of any non-constant polynomial $A: \F_p \to \F_p$ with degree at most $t$, then $P$ coincides on $S^n$ with a polynomial that in particular has bounded degree-$\lfloor d/(t+1) \rfloor$-rank in the sense of Green and Tao. Similarly, we prove that if the assumption holds even for $t=d$, then $P$ coincides on $S^n$ with a polynomial determined by a bounded number of coordinates.

\end{abstract}

\tableofcontents

\section{Introduction and main results}

Throughout this paper, the letter $n$ will always denote a positive integer, and all our statements will be uniform in $n$.

A landmark result of Green and Tao proved in 2007 \cite{Green and Tao} states that over a finite prime field $\F_p$ for some prime $p$, a multivariate polynomial with degree $1 \le d < p$ that is not approximately equidistributed can be expressed as a function of a bounded number of polynomials each with degree at most $d-1$. More formally, we have the following statement.

\begin{theorem}[\cite{Green and Tao}, Theorem 1.7] \label{formulation of Green and Tao}

Let $p$ be a prime, and let $1 \le d < p$ be an integer. Then there exists a function $K_{p,d}: (0, 1 \rbrack \rightarrow \N$ such that for every $\e>0$, if $P: \F_p^n \rightarrow \F_p$ is a polynomial with degree $d$ satisfying $|\E_{x \in \F_p^n} \omega_p^{s P(x)}| \ge \epsilon$ for some $s \in \F_p^*$, then there exist $k \le  K_{p,d}(\epsilon)$, polynomials $P_1, \dots, P_k: \F_p^n \to \F_p$ with degree at most $d-1$ and a function $F:\F_p^k \to \F_p$ satisfying \[P = F(P_1, \dots, P_k).\]

\end{theorem}

It has been known since at least the works of Janzer \cite{Janzer} and Mili\'cevi\'c \cite{Milicevic} that the conclusion can be made qualitatively more precise. Before stating this strengthening, let us define a notion of degree-$d$ rank for polynomials.

\begin{definition} \label{rank of a polynomial} Let $\F$ be a field, and let $P: \F^n \rightarrow \F$ be a polynomial. Let $d \ge 1$ be an integer.

We say that a polynomial $P$ has \emph{degree-$d$ rank at most $1$} if we can write $P$ as a product of polynomials each with degree at most $d$.

The \emph{degree-$d$ rank} of $P$ is defined to be the smallest nonnegative integer $k$ such that there exist polynomials $P_1, \dots, P_k$ each with degree-$d$ rank at most $1$, with degree at most the degree of $P$, and satisfying \[P = P_1 + \dots + P_k.\] We denote this quantity by $\rk_d P$.

\end{definition}

The zero polynomial in particular has degree-$d$ rank equal to $0$ for all $d$, and constant polynomials have degree-$d$ rank at most $1$ for all $d$.

We define this notion of degree-$d$ rank in this way as doing so will be convenient for us, but it is worth pointing out that in the original paper \cite{Green and Tao} of Green and Tao, the notion referred to as the degree-$d$ rank was slightly different: for instance the degree-$(d-1)$ rank was the largest possible $k$ in Theorem \ref{formulation of Green and Tao}. Nonetheless, it follows immediately from the definitions that for every polynomial $P:\F_p^n \to \F_p$ and every positive integer $d$, the degree-$d$ rank of $P$ in the sense of Green and Tao is at most $d$ times the degree-$d$ rank of $P$ in our sense.  Therefore, proving that a polynomial has bounded degree-$d$ rank in our sense implies showing that it has bounded degree-$d$ rank in the sense of Green and Tao.

The main qualitative refinement shown in the papers of Janzer \cite{Janzer} and Mili\'cevi\'c \cite{Milicevic} is that there exists some function $H_{p,d}: (0, 1] \to \N$ such that under the assumptions of Theorem \ref{formulation of Green and Tao}, we can find $k \le H_{p,d}(\epsilon)$ and polynomials $Q_1, R_1, \dots, Q_k, R_k$ satisfying \[\deg Q_i, \deg R_i \le d-1 \text{ and } \deg Q_i + \deg R_i \le d\] for each $i \in [k]$ and such that \[P = Q_1 R_1 + \dots + Q_k R_k.\] In other words, it was shown that \[\rk_{d-1} P \le H_{p,d}(\epsilon).\] This is a bound on the degree-$(d-1)$ rank of $P$, and the numerous developments which arose out of Theorem \ref{formulation of Green and Tao} have to our knowledge entirely or almost entirely focused on the degree-$(d-1)$ rank of $P$: some extended the range of validity of the results (Kaufman and Lovett \cite{Kaufman and Lovett}, Bhowmick and Lovett \cite{Bhowmick and Lovett}), and others improved the quantitative bounds on the degree-$(d-1)$ rank, through the closely related question of comparing the partition rank to the analytic rank of tensors (Janzer \cite{Janzer}, Mili\'cevi\'c \cite{Milicevic}, Adiprasito, Kazhdan and Ziegler \cite{Adiprasito Kazhdan Ziegler}, Moshkovitz and Cohen \cite{Moshkovitz and Cohen 1}, \cite{Moshkovitz and Cohen 2}, Moshkovitz and Zhu \cite{Moshkovitz and Zhu}).

For the purposes of studying approximate equidistribution of polynomials this is unsurprising, since the notion of degree-$(d-1)$ rank is indeed by far the most relevant: for instance a random polynomial of the type \[x_1 Q(x_2, \dots, x_n)\] with $\deg Q = d-1$ has high degree-$(d-2)$ rank but is nonetheless not approximately equidistributed, since the probability that it takes the value $0$ is approximately $2/p-1/p^2 > 1/p$.

Rather than focus on the fact that for a degree-$d$ polynomial, lack of equidistribution implies bounded degree-$(d-1)$ rank, we may ask for analogues of this statement involving much stronger properties in the assumption and in the conclusion. Correspondingly, the main motivations of this paper are twofold. In one direction, we ask what can be deduced about polynomials for which we know much more than lack of equidistribution. What can we say if we know that a polynomial does not take every value of $\F_p$, or has a smaller range still, in a sense to be made precise ? In the other direction, we can ask, for a fixed integer $1 \le e \le d-1$, whether there are any properties of the distribution of the values of a polynomial which would guarantee that its degree-$e$ rank is bounded above. We will contribute to both directions simultaneously, by showing that if a polynomial $P$ does not have full range, then it must have bounded degree-$e$ rank, for some integer $e$ that is determined by the degree of $P$ and by the smallest degree of a non-constant \emph{one-variable} polynomial that has a range contained in the range of $P$.

\begin{theorem} \label{Main result on ranges with full alphabet} Let $p$ be a prime, and let $1 \le t \le d < p$ be integers. There exists a positive integer $\gamma(p,d,t)$ such that the following holds. Let $P: \mathbb{F}_p^n \to \mathbb{F}_p$ be a polynomial with degree at most $d$. Assume that the image $P(\F_p^n)$ does not contain the image of $\mathbb{F}_p$ by any non-constant polynomial $\mathbb{F}_p \to \mathbb{F}_p$ with degree at most $t$.

\begin{enumerate}

\item If $t \le d-1$, then $P$ has degree-$\lfloor d/(t+1) \rfloor$-rank at most $\gamma(p,d,t)$.

\item If $t = d$ then $P$ is a constant polynomial.

\end{enumerate}

\end{theorem}

The value $\lfloor d/(t+1) \rfloor$ in the degree of the rank in Theorem \ref{Main result on ranges} is optimal in general, as the following example shows.

\begin{example}

Let $p$ be a prime, let $1 \le d < p$ be an integer, let $t,u \ge 1$ be integers such that $tu \le d$. If $Q$ is a random polynomial $Q: \F_p^n \to \F_p$ with degree $u$, then $Q^t$ has degree at most $d$, the image $Q(\F_p^n)$ is contained in the set $\{y^t: y \in \F_p\}$ of $t$-th power residues mod-$p$, but the degree-$(u-1)$-rank of $P$ is usually arbitrarily large as $n$ tends to infinity, even if it is taken in the sense of Green and Tao.

\end{example}

The last part follows from a counting argument: as n tends to infinity there are $p^{O(n^{u-1})}$ polynomials $\F_p^n \to \F_p$ with degree at most $u-1$, so for every $k \ge 1$, the number of polynomials of the type $F(P_1,\dots,P_k)$ with $P_1, \dots, P_k$ with degree at most $u-1$ and $F: \F_p^k \to \F_p$ a function is at most $p^{p^k} p^{O(kn^{u-1})} = p^{O(kn^{u-1})}$, whereas there are $p^{\Omega(kn^u)}$ polynomials $\F_p^n \to \F_p$ with degree $u$ and hence at least $1/t$ times as many polynomials of the type $Q^t$ above.

Powers of polynomials are not the only simple examples of polynomials that do not have full range. They can instead be viewed as a special case of a broader class of examples that arises from composition with a one-variable polynomial.

\begin{example}

Let $p$ be a prime, let $1 \le d < p$ be an integer, let $t,u \ge 1$ be integers such that $tu \le d$. If $Q: \F_p^n \to \F_p$ is a polynomial with degree $u$, and $A: \F_p \to \F_p$ is a polynomial with degree $t$, then the polynomial $A \circ Q$ has degree at most $d$, and the image $A \circ Q(S^n)$ is contained in the image $A(\F_p)$.

\end{example}

We stress that the main result from the approximate equidistribution regime will itself be an important black box that we will use in our proof of Theorem \ref{Main result on ranges with full alphabet}. We shall in fact prove results in a more general setting than that of Theorem \ref{Main result on ranges with full alphabet}, where we allow the assumption to be on the image $P(S^n)$ for some non-empty subset $S$ of $\F_p$ rather than on the whole image $P(\F_p^n)$. On a first reading the set $S$ may be taken to be $\{0,1\}$. In the setting of restrictions to $S^n$, the approximate equidistribution statement was proved by Gowers and the author in \cite{Gowers and K equidistribution}. Before stating it, let us recall from that paper two points to be aware of regarding restrictions of polynomials to $S^n$.

The first is that whereas an affine polynomial is either constant or perfectly equidistributed on $\F_p^n$, there is already something to say about the distribution of an affine polynomial $P$ on $S^n$ for general non-empty $S$: if $S \neq \F_p$ and $P$ depends only on one coordinate, then $P(S^n)$ is not even the whole of $\F_p$. As a simple Fourier argument however shows (\cite{Gowers and K equidistribution}, Proposition 2.2), an affine polynomial depending on many coordinates is approximately equidistributed on $S^n$, provided that $S$ contains at least two elements. The second is that we may no longer hope to conclude in general that a polynomial with degree $d$ which is not approximately equidistributed on $S^n$ must itself have bounded degree-$(d-1)$ rank: for instance, the polynomial \[\sum_{i=1}^n x_i^2 - x_i\] has degree-$1$ rank equal to $n$, but only takes the value $0$ on $\{0,1\}^n$ and is in particular not approximately equidistributed on $\{0,1\}^n$. Nonetheless, the zero polynomial, with which this polynomial coincides on $\{0,1\}^n$, itself has degree-$1$ rank equal to $0$.

These two remarks motivate an extension of Definition \ref{rank of a polynomial}.

\begin{definition}\label{further definitions of rank of a polynomial}

Let $\F$ be a field, and let $P: \F^n \rightarrow \F$ be a polynomial.

The \emph{degree-$0$ rank} of $P$ is defined to be the smallest nonnegative integer $k$ such that we can write $P$ as a linear combination of at most $k$ monomials. We denote this quantity by $\rk_{0} P$.

If $d$ is a nonnegative integer and $S$ is a non-empty subset of $\F$ then we define the \emph{degree-$d$ rank of $P$ with respect to $S$} as the smallest value of $\rk_d (P-P_0)$, where the minimum is taken over all polynomials $P_0$ with degree at most the degree of $P$ and satisfying $P_0(S^n) = \{0\}$. We denote this quantity by $\rk_{d,S} P$.

\end{definition}

We now recall a slight weakening of the main result of \cite{Gowers and K equidistribution}, Theorem 1.4 from that paper. (Although the full statement of that theorem is slightly more precise, the formulation below is slightly simpler to use and suffices for the purposes of the present paper.)

\begin{theorem} \label{Equidistribution on restricted alphabets} Let $p$ be a prime, let $1 \le d < p$ be an integer, and let $S$ be a non-empty subset of $\F_p$. There exists a function $H_{p,d,S}: (0, 1 \rbrack \rightarrow \N$ such that for every $\e>0$, if $P: \F_p^n \rightarrow \F_p$ is a polynomial with degree $d$ satisfying $|\E_{x \in S^n} \omega_p^{s P(x)}| \ge \epsilon$ for some $s \in \F_p^*$, then $ \rk_{d-1,S} P \le H_{p,d,S}(\epsilon)$. \end{theorem}

We note that if $S$ has size $1$, then Theorem \ref{Equidistribution on restricted alphabets} as well as many of the new results of the present paper hold true for immediate reasons: the set $S^n$ then also has size $1$, so every polynomial coincides on $S^n$ with a constant polynomial, so has degree-$d$ rank at most $1$ for every $d$.

Theorem \ref{Equidistribution on restricted alphabets} as well as the main results in this paper with an assumption on the distribution of values of a polynomial on $S^n$ have a conclusion which holds up to some polynomial that vanishes on $S^n$ (and that has degree at most the degree of the original polynomial). Although we have by now explained why having a conclusion that merely holds up to these polynomials cannot be avoided, we point out that that it is not difficult to describe polynomials which vanish on $S^n$: for instance, if $S = \{0,1\}$, then all powers $x_i, x_i^2, x_i^3, \dots$ are the same, and a polynomial which collapses to $0$ after we identify $x_i^{a_i}$ with $x_i$ inside all monomials whenever $i \in [n]$ and $a_i \ge 1$ necessarily vanishes on $\{0,1\}^n$. Conversely, the combinatorial Nullstellensatz shows that if the polynomial resulting from this identification is not the zero polynomial, then it does not vanish on $\{0,1\}^n$. This characterisation extends, with the exact same proof, to arbitrary non-empty subsets $S$ of $\F_p$: a polynomial vanishes on $S^n$ if and only if substituting $x_i^{a_i}$ by its remainder modulo \[\prod_{w \in S} (x_i-w)\] at every instance of $x_i^{a_i}$, for all $i \in [n]$ and $a_i \ge 1$ leads to the zero polynomial after expanding and simplifying.

When $S$ is not the whole of $\F_p$, one important difference between the sets $\F_p^n$ and $S^n$ is that the former is invariant under linear transformations, whereas the latter is not. We have already discussed one effect on this: the fact that $x_1$ does not take every value of $\F_p$ whereas $x_1 + \dots + x_n$ is approximately equidistributed for $n$ large. The role of coordinates as opposed to general degree-$1$ polynomials will manifest itself further in the proofs and in the main results of this paper. For this purpose let us make one last round of definitions.

\begin{definition}

Let $p$ be a prime, and let $P:\F_p^n \to \F_p$ be a polynomial. 

For each $i \in [n]$, we say that $P$ depends on $x_i$ if the coordinate $x_i$ appears in one of the monomials of $P$. 

We write $I(P)$ for the set of $i \in [n]$ such that $P$ depends on $x_i$. We will refer to this set as the \emph{set of coordinates on which $P$ depends.}

If $P$ is an affine polynomial, that is, a polynomial with degree at most $1$, then we will also write $I(P)$ as $Z(P)$, and refer to it as the \emph{support} of $P$.

For $k$ nonnegative integer, we will say that $P$ is \emph{$k$-determined} if the set $I(P)$ has size at most $k$.

\end{definition}

Using Theorem \ref{Equidistribution on restricted alphabets} as a black box we will prove the following analogue of Theorem \ref{Main result on ranges with full alphabet}, where the assumption on the image is now on $P(S^n)$ rather than on $P(\F_p^n)$. The following theorem is the main result that we shall prove in the present paper.

\begin{theorem} \label{Main result on ranges} Let $p$ be a prime, let $1 \le t \le d < p$ be integers and let $S$ be a non-empty subset of $\F_p$. Then there exists a positive integer $C(p,d,t)$ such that the following holds. Let $P: \mathbb{F}_p^n \to \mathbb{F}_p$ be a polynomial with degree at most $d$. Assume that $P(S^n)$ does not contain the image of $\mathbb{F}_p$ by any non-constant polynomial $\mathbb{F}_p \to \mathbb{F}_p$ with degree at most $t$. \begin{enumerate}

\item If $t \le d-1$, then $P$ coincides on $S^n$ with a polynomial that has degree-$\lfloor d/(t+1) \rfloor$-rank at most $C(p,d,t)$ and has degree at most $d$.

\item If $t = d$ then $P$ coincides on $S^n$ with a linear combination of at most $C(p,d,t)$ monomials with degrees at most $d$.

\end{enumerate} Equivalently, in both cases we have \[ \rk_{\lfloor \frac{d}{t+1} \rfloor, S} P \le C(p,d,t).\]

\end{theorem}

The optimal bounds in Theorem \ref{Main result on ranges} and in several of our other statements involving the set $S$ may depend on the choice of $S$. However, to avoid heavy notation we will at many places avoid making this dependence explicit. (We may safely do so, since for each prime $p$ there are only finitely many subsets of $\F_p$).

Let us look at the extreme cases of item 1 from Theorem \ref{Main result on ranges}, and at a situation where they are both simultaneously realised.

\begin{corollary} \label{corollary of main result}

Let $p$ be a prime, let $1 \le t \le d < p$ be integers and let $S$ be a non-empty subset of $\F_p$. Let $P:\F_p^n \to \F_p$ be a polynomial with degree at most $d$. Let $C_{(i)}(p,d) = C(p,d,1)$ and let $C_{(ii)}(p,d) = C(p,d, \lfloor d/2 \rfloor)$. \begin{enumerate}[(i)]

\item If $P(S^n) \neq \F_p$ then $\rk_{\lfloor d/2 \rfloor,S} P \le C_{(i)}(p,d)$.

\item If $P(S^n)$ does not contain the image of any non-constant polynomial $\F_p \to \F_p$ with degree at most $\lfloor d/2 \rfloor$ then $\rk_{1,S} P \le C_{(ii)}(p,d)$.

\item If $d=3$ and $P(S^n) \neq \F_p$, then $\rk_{1,S} P \le C_{(i)}(p,3) = C_{(ii)}(p,3)$.

\end{enumerate}

\end{corollary}

\begin{proof} Items (i) and (ii) follow from taking $t=1$ and $t = \lfloor d/2 \rfloor$ in Theorem \ref{Main result on ranges} respectively. Item (iii) follows from either of the items (i) and (ii). \end{proof}

We make the assumption $d<p$ in Theorem \ref{Main result on ranges} in order to be able to prove it using Theorem \ref{Equidistribution on restricted alphabets}, which assumes $d<p$. However, this assumption in the original result of Green and Tao, Theorem \ref{formulation of Green and Tao}, was later removed by Kaufman and Lovett \cite{Kaufman and Lovett}, and it seems likely to us that Theorem \ref{Equidistribution on restricted alphabets} is true without this assumption as well. Aside from the use of Theorem \ref{Equidistribution on restricted alphabets}, the remainder of our proof of Theorem \ref{Main result on ranges} does not use this assumption at any point. In other words, our proof shows that removing this assumption in Theorem \ref{Equidistribution on restricted alphabets} would suffice to remove it in Theorem \ref{Main result on ranges} (and to remove it in Theorem \ref{Main result on ranges with full alphabet} and in Corollary \ref{corollary of main result}). 

If Theorem \ref{Main result on ranges} still holds for $d \ge p$, then we would also be able to take $t \ge p$ in Theorem \ref{Main result on ranges}, but if the latter inequality is satisfied (or even if $t \ge p-1$ is) then the content of Theorem \ref{Main result on ranges} becomes immediate. Indeed Fermat's little theorem allows us to construct a degree-$(p-1)$ polynomial $\F_p \to \F_p$ with image equal to any prescribed pair of values, and the assumption of Theorem \ref{Main result on ranges} is then only satisfied by polynomials which coincide with a constant on $S^n$.

We now turn our attention to the case of degree-$2$ polynomials. Throughout the paper, we will write $Q_p$ for the set $\{y^2: y \in \F_p\}$ of mod-$p$ quadratic residues. Provided that $p \ge 3$, this set has size $\frac{p+1}{2}$ and is in particular not the whole of $\F_p$. We say that a subset of $\F_p$ is an \emph{affine translate} of $Q_p$ if it can be written as $aQ_p+b$ for some $a \in \F_p^*$ and some $b \in \F_p$. In light of the preceding discussion we can formulate three basic constructions of a degree-$2$ polynomial $P$ such that $P(S^n) \neq \F_p$. \begin{enumerate}[(i)]

\item A polynomial of the type $A \circ L$ for some affine polynomial $L: \F_p^n \to \F_p$ and some degree-$2$ polynomial $A:\F_p \to \F_p$. (Equivalently, the sum of a multiple of $L^2$ and of a constant.) 

\item A polynomial that depends only on a small number $r < \log p / \log |S|$ of coordinates, since $P(S^n)$ then necessarily has size at most $|S|^r$.

\item A polynomial that vanishes on $S^n$ and has degree at most $2$.

\end{enumerate}

The first item of the following result can be interpreted as a converse which says that every example arises as a sum of these three examples, letting aside the value of the bound on the number of coordinates in the second example.

\begin{proposition}\label{Degree-2 result} There exists an absolute constant $\kappa > 0$ such that the following holds. Let $p$ be a prime, and let $S$ be a non-empty subset of $\F_p$. Let $P:\F_p^n \to \F_p$ be a polynomial with degree $2$.

\begin{enumerate}

\item If $P(S^n) \neq \F_p$, then there exists an affine polynomial $L: \F_p^n \to \F_p$, a degree-$2$ polynomial $A: \F_p \to \F_p$, and a $\kappa p^{15}$-determined polynomial $J$ with degree at most $2$ such that $P$ coincides on $S^n$ with $A \circ L + J$. (Equivalently, with $AL^2 + J$ for some $A \in \F_p$, with $J$ changed by a constant.)

\item If furthermore $P(S^n)$ does not contain any affine translate of $Q_p$, then $P$ coincides on $S^n$ with a $\kappa p^{15}$-determined polynomial that has degree at most $2$.

\end{enumerate}

\end{proposition}

Item 1 from Proposition \ref{Degree-2 result} is significantly stronger than the conclusion that item 1 from Theorem \ref{Main result on ranges} gives in the corresponding case $d=2$ and $t=1$: the latter is merely that $P$ has bounded degree-$1$ rank with respect to $S$, which we already know by Theorem \ref{Equidistribution on restricted alphabets}. The proof of Proposition \ref{Degree-2 result} will instead use different techniques which do not appear to generalise well to higher-degree polynomials.

In the more general case where $P$ has general degree $2 \le d \le p-1$, one may ask whether just as with item 1 from Proposition \ref{Degree-2 result}, it is the case that provided that $P(S^n) \neq \F_p$ we can always obtain a decomposition $P = A \circ Q + J$ with $Q: \F_p^n \to \F_p$, $A: \F_p \to \F_p$ polynomials satisfying $\deg Q \deg A \le d$, $\deg A \ge 2$ and with $J$ a polynomial determined by a bounded number of coordinates and with degree at most $d$. This is however not the case in general, as a wider class of examples comes in: for instance, if $d=p-1$, then the polynomial $A:x \to x^{p-1}$ satisfies $A(\F_p) = \{0,1\}$, so if $L_1, \dots, L_{p-2}$ are arbitrary affine polynomials then the image of $\F_p^n$ by the polynomial \[P = A \circ L_1 + \dots + A \circ L_{p-2}\] does not contain $p-1$. For $d=2$, such a situation cannot occur, because the Cauchy-Davenport theorem, which will play some role in the proof of Proposition \ref{Degree-2 result}, shows that the sumset of any two affine translates of $Q_p$ is the whole of $\F_p$. (However, this is by no means the only or even the main specificity of the case $d=2$ that allows us to say more there than for general $d<p$.)

The special case $t=d$ of Theorem \ref{Main result on ranges} has a weaker variant where we merely forbid images of $S$ (rather than of $\F_p$) by one-variable polynomials; its proof is then much simpler, and already contains one of the main ideas of the proof of Theorem \ref{Main result on ranges}.

\begin{proposition}\label{related statement} Let $p$ be a prime, let $d \ge 1$ be an integer, and let $S$ be a non-empty subset of $\F_p$. If $P:\F_p^n \to \F_p$ is a polynomial with degree at most $d$ and $P(S^n)$ does not contain the image of $S$ by any non-constant polynomial $A:\F_p \to \F_p$ with degree at most $d$ then $P$ coincides on $S^n$ with a constant polynomial. \end{proposition}

\begin{proof} We can write \[P(x) = \sum_{k=0}^d P_k(x_1, \dots, x_{n-1}) x_n^k\] for some polynomials $P_0, \dots, P_d: \F_p^{n-1} \to \F_p$ and for all $x \in \F_p^n$. The assumption shows that every polynomial $P_k$ with $1 \le k \le d$ vanishes on $S^n$, so $P$ coincides on $S^n$ with a polynomial with degree at most $d$ and which only involves the coordinates $x_1, \dots, x_{n-1}$. We then conclude by induction on $n$. \end{proof}

When $S$ has size $2$, checking Proposition \ref{related statement} is immediate: for every subset $U$ of $\F_p$ with size $2$ there is an affine polynomial $A$ such that $A(S) = U$, so the assumption states that $P(S^n)$ has size $1$, which is the same as the conclusion.

We note that by taking $S=\F_p$ in Proposition \ref{related statement} we obtain item 2 of Theorem \ref{Main result on ranges with full alphabet}. Since item 1 from Theorem \ref{Main result on ranges} specialises to item 1 from Theorem \ref{Main result on ranges with full alphabet}, establishing Theorem \ref{Main result on ranges} suffices to prove Theorem \ref{Main result on ranges with full alphabet}.

The basic strategy which we will use to prove Theorem \ref{Main result on ranges} will be essentially as follows: because $P$ has degree $d$ and is not approximately equidistributed, Theorem \ref{Equidistribution on restricted alphabets} shows that $P$ coincides on $S^n$ with some polynomial with degree at most $d$ and of the type \[M \circ (P_1,\dots,P_k)\] where $k$ is bounded and $M$ is some polynomial. We then aim to imitate to some extent the proof of Proposition \ref{related statement}, with the polynomials $P_1, \dots, P_k$ playing the role of the coordinates $x_1, \dots, x_n$. Unlike the latter, the former are however not jointly equidistributed in general (with the set of \emph{values} taken to be $\F_p^k$ rather than $S^n$), so we cannot conclude straight away as we did in the proof of Proposition \ref{related statement}, but one of the following is always true: either the polynomials are approximately jointly equidistributed, in which case the image $(P_1, \dots, P_k)(S^n)$ is the same as if the polynomials $P_1, \dots, P_k$ were jointly equidistributed, or they are not, in which case at least one non-trivial linear combination of the polynomials $P_1, \dots, P_k$ has bounded degree-$(d'-1)$ rank with respect to $S$, where \[d' = \max(\deg P_1, \dots, \deg P_k),\] and we may hence without loss of generality assume that $P$ coincides on $S^n$ with some polynomial with degree at most $d$ and of the type \[M' \circ (P_1, \dots, P_{k-1}, Q_1, \dots, Q_{k'})\] where $k'$ is bounded, $Q_1, \dots, Q_{k'}$ are polynomials with degree strictly smaller than the degree of $P_k$, and $M'$ is some polynomial. This second step, in turn, can only be performed a bounded number of times, which will conclude the argument.

The remainder of the paper is organised as follows. In Section \ref{Section: A dichotomy on the conditional range and rank of polynomials} we will state and prove the dichotomy on polynomials on which the previous argument relies, and which will be used in Section \ref{Section: Deduction of the more precise result for degree-2 polynomials} and Section \ref{Section: Deduction of the main theorem in the general case}. Then, in Section \ref{Section: Deduction of the more precise result for degree-2 polynomials}, we will prove Proposition \ref{Degree-2 result}, and in Section \ref{Section: Deduction of the main theorem in the general case} we will write the details of the process described in the previous paragraph and hence prove Theorem \ref{Main result on ranges}.

\section*{Acknowledgement}

We thank Lisa Sauermann for explaining to us the proof of Lemma \ref{Lower bound lemma for polynomials} in a special case already involving the main arguments of the proof of this lemma. This proof is more elegant than the proof that we would otherwise have used and furthermore provides better bounds.

\section{A dichotomy on the conditional range and rank of polynomials}\label{Section: A dichotomy on the conditional range and rank of polynomials}

Let us begin by formally stating the dichotomy which we just mentioned. The purpose of the remainder of the present section will be to prove it, and the two auxiliary lemmas which we will use to do so, Lemma \ref{Multidimensional equidistribution on restricted alphabets} and Lemma \ref{Lower bound lemma for polynomials}, will not themselves be directly used at any point in any later section of the paper.

\begin{proposition} \label{Conditional full range on restricted alphabets} Let $p$ be a prime, let $1 \le d < p$ be an integer, let $S$ be a non-empty subset of $\F_p$, and let $k$ be a nonnegative integer. Let $P_1, \dots, P_k, P: \F_p^n \to \F_p$ be polynomials with degree at most $d$. If \[\rk_{d-1,S} P > H_{p,d,S}(1/2p)\] then $P(S^n) = \F_p$. If $k \ge 1$ and \[\rk_{d-1,S} \left( P + \sum_{i=1}^k a_i P_i \right) > H_{p,d,S}(|S|^{-(p-1)dk} / 2p) \] for every $a \in \F_p^k$, then for every $(v_1, \dots, v_k) \in \F_p^k$ the following holds: if $(v_1, \dots, v_k) \in (P_1, \dots, P_k)(S^n)$ then \[\F_p \times \{(v_1, \dots, v_k)\} \subset (P,(P_1, \dots, P_k))(S^n).\]

\end{proposition}

We begin by showing a statement which informally says that if $P$ is sufficiently far from any linear combination of a family of polynomials $P_1, \dots, P_k$, then in absolute terms the values of $P(x)$ and of $(P_1(x), \dots, P_k(x))$ are approximately independent, when $x$ is chosen uniformly at random in $S^n$.

\begin{lemma} \label{Multidimensional equidistribution on restricted alphabets} Let $p$ be a prime, let $1 \le d < p$ be an integer, let $S$ be a non-empty subset of $\F_p$ and let $k \ge 1$ be an integer. Let $P_1, \dots, P_k, P: \F_p^n \to \F_p$ be polynomials with degree at most $d$. Let $\epsilon > 0$. If \[ \rk_{d-1,S} \left( P + \sum_{i=1}^k a_i P_i \right) > H_{p,d,S}(\e)\] for every $(a_1, \dots, a_k) \in \F_p^k$, then for every $u \in \F_p$ and $(v_1, \dots, v_k) \in \F_p^k$ we have \[ |\P_{x \in S^n}(P(x) = u, P_1(x) = v_1, \dots, P_k(x) = v_k) - p^{-1} \P_{x \in S^n}(P_1(x) = v_1, \dots, P_k(x) = v_k)| \le \e. \]

\end{lemma}

\begin{proof}

The proof is a standard Fourier-analytic calculation. The difference in the left-hand side can be rewritten as \[ \E_{a, a_1, \dots, a_k \in \F_p} \E_{x \in S^n} \omega_p^{aP(x)+a_1P_1(x) + \dots + a_kP_k(x)} - p^{-1} \E_{a_1, \dots, a_k \in \F_p} \E_{x \in S^n} \omega_p^{a_1P_1(x) + \dots + a_kP_k(x)}\] \[= p^{-(k+1)} \sum_{a \in \F_p^*, a_1, \dots, a_k \in \F_p} \E_{x \in S^n} \omega_p^{aP(x)+a_1P_1(x) + \dots + a_kP_k(x)}.\] By Theorem \ref{Equidistribution on restricted alphabets} each Fourier coefficient (i.e. each inner expectation) has modulus at most $\epsilon$. We conclude by the triangle inequality. \end{proof}

We note that Lemma \ref{Multidimensional equidistribution on restricted alphabets} in particular implies a ``multidimensional" version of Theorem \ref{Equidistribution on restricted alphabets}: if a family of polynomials $P_1, \dots, P_k$ each with degree at most $d$ is such that every non-zero linear combination of $P_1, \dots, P_k$ has high degree-$(d-1)$ rank with respect to $S$, then the $k$-tuple $(P_1(x), \dots, P_k(x))$ is approximately equidistributed on the set of values $\F_p^k$ when $x$ is chosen uniformly at random in $S^n$. Although that multidimensional version guided us in discovering the proof of Theorem \ref{Main result on ranges}, we will however not formally use it at any point in this paper, so we leave it to the interested reader to write the deduction in full.

Lemma \ref{Multidimensional equidistribution on restricted alphabets} stated an approximation in terms of absolute probabilities: however it does not by itself immediately entail that $P(x)$ is approximately equidistributed (or even that it takes every value of $\F_p$, which is what ultimately interests us) conditionally on any $k$-tuple of values of $(P_1(x), \dots, P_k(x))$ taken with positive probability on $S^n$: if the probability \[\P_{x \in S^n}(P_1(x)=v_1, \dots, P_k(x)=v_k)\] is positive but very small, then the probability \[\P_{x \in S^n}(P(x)=u, P_1(x)=v_1, \dots, P_k(x)=v_k)\] could still potentially be zero. 

In our next lemma we show that this cannot happen. We thank Lisa Sauermann for explaining to us its proof in the case where all polynomials have degree $1$ and $S = \{0,1\}$.

\begin{lemma} \label{Lower bound lemma for polynomials} Let $p$ be a prime, let $ d \ge 1$ be an integer, and let $S$ be a non-empty subset of $ \F_p$. Let $k \ge 1$ be an integer and let $ P_1, \dots, P_k$ be polynomials $\F_p^n \rightarrow \F_p$ each with total degree at most $ d$. Then for each $v \in \F_p^k$ the probability \[\mathbb{P}_{x \in S^n}(P_1(x)=v_1, \dots, P_k(x) = v_k)\] is either $ 0$ or at least $ |S|^{-(p-1)dk}$. \end{lemma}

\begin{proof} Let $v=(v_1, \dots, v_k) \in \F_p^k$ be fixed throughout. We define the polynomial \[ E:= \prod\limits_{i=1}^k ((P_i - v_i)^{p-1} - 1).\] Then for each $x \in S^n$ we have $ E(x) \neq 0$ if and only if \begin{equation} (P_1(x), \dots, P_k(x)) = (v_1, \dots, v_k). \label{value of k-tuple} \end{equation} The polynomial $E$ has total degree at most $(p-1)dk$ and can hence be expanded as a linear combination of monomials $x_1^{t_1} \dots x_n^{t_n}$ with $t_1 + \dots + t_n$ at most $(p-1)dk$. All elements of $S$ are roots of the polynomial $\Delta_S: \F_p \rightarrow \F_p$ defined by \[ \Delta_{S}(y) = \prod\limits_{w \in S} (y-w),\] so any polynomial $A:\F_p \rightarrow \F_p$ in one variable coincides on $S$ with its remainder $A \mod \Delta_{S}$ through the Euclidean division by $ \Delta_{S}$, and hence each monomial $ x_1^{t_1} \dots x_n^{t_n}$ coincides on $S^n$ with the product of remainders \[ (x_1^{t_1} \text{ mod } \Delta_{S}(x_1)) \dots (x_n^{t_n} \text{ mod } \Delta_{S}(x_n)),\] which we expand into a linear combination of monomials of the type $ x_1^{t_1} \dots x_n^{t_n}$ with $0 \le t_i\le |S|-1$ for each $ i \in \lbrack n \rbrack$. We have hence obtained a polynomial $ R$ in $ x_1, \dots x_n$ which coincides with $ E$ on $S^n$ and has each of its monomials of the type $ x_1^{t_1} \dots x_n^{t_n}$ with $0 \le t_i \le |S|-1$ for each $ i \in \lbrack n \rbrack$ and has total degree at most the total degree of $E$. Because the degree of each monomial of $ E$ is at most $ (p-1)dk$, the total degree of $ R$ is at most $ (p-1)dk$.

For every $x \in S^n$, the three statements \[(P_1(x), \dots, P_k(x)) \neq (v_1, \dots, v_k) \text{, } E(x) = 0 \text{, } R(x) = 0\] are all equivalent. In particular, if there exists $x \in S^n$ satisfying \eqref{value of k-tuple} then $ R$ is not the zero polynomial, so one of its monomials with degree equal to the total degree of $R$ (which may be the constant term, if $R$ has degree $0$) has a non-zero coefficient. Without loss of generality we can assume that this monomial is $ x_1^{t_1} \dots x_s^{t_s}$ with $ 0 \le s \le (p-1)dk$ and $ 1 \le t_i \le |S|-1$ for each $ i \in \lbrack s \rbrack$. For every $u:=(u_{s+1}, \dots, u_n) \in \F_p^{\lbrack s+1, n \rbrack}$ the polynomial $R_u: \F_p^s \rightarrow \F_p$ defined by \[R_u(x_1, \dots, x_s) = R(x_1, \dots, x_s, u_{s+1}, \dots, u_n)\] still has total degree $s$. By the combinatorial Nullstellensatz, there exists $ (x_1, \dots, x_s) \in S^s$ such that $ R_u(x_1, \dots, x_s) \neq 0$ and hence such that \[(P_1, \dots, P_k)(x_1, \dots, x_s, u_{s+1}, \dots, u_n) = (v_1, \dots, v_k).\] By the law of total probability we conclude that \[\mathbb{P}_{x \in S^n}(P_1(x)=v_1, \dots, P_k(x) = v_k) \ge |S|^{-s} \ge |S|^{-(p-1)dk}. \qedhere \]  \end{proof}

Although we will not use this we note that these bounds are sharp for $S = \{0,1\}$, as can be seen by taking \[P_1 = x_1 \dots x_d, P_2 = x_{d+1} \dots x_{2d}, \dots, P_k = x_{(k-1)d+1} \dots x_{kd} \text{ and } (v_1, \dots, v_k) = (1, \dots, 1).\]

The proof of Proposition \ref{Conditional full range on restricted alphabets} is now almost finished. The first part of Proposition \ref{Conditional full range on restricted alphabets} follows immediately from Theorem \ref{Equidistribution on restricted alphabets}. In the case $k \ge 1$, for every fixed $(v_1, \dots, v_k) \in (P_1, \dots, P_k)(S^n)$, applying Lemma \ref{Multidimensional equidistribution on restricted alphabets} with $\epsilon = p^{-1} |S|^{-(p-1)dk} / 2$ and then Lemma \ref{Lower bound lemma for polynomials} we obtain \[ \P_{x \in S^n}(P(x)=u, P_1(x)=v_1, \dots, P_k(x)=v_k) \ge p^{-1} |S|^{-(p-1)dk} / 2 > 0.\]

\section{Deduction of the more precise result for degree-2 polynomials}\label{Section: Deduction of the more precise result for degree-2 polynomials}

In this section we prove Proposition \ref{Degree-2 result}. The proof will use an inductive argument revolving around the following type of decomposition.

\begin{definition}

Let $p$ be a prime, and let $P:\F_p^n \to \F_p$ be a degree-$2$ polynomial. For any nonnegative integers $k$,$l$, we say that a \emph{$k$-square-$l$-determined} decomposition of $P$ is a decomposition of the type \begin{equation} P = A_1 L_1^2 + \dots + A_k L_k^2 + J \label{square-determined decomposition} \end{equation} for some $A_1, \dots, A_k \in \F_p$, some affine polynomials $L_1, \dots, L_k$, and some $l$-determined polynomial $J$ with degree at most $2$.

\end{definition}

Throughout this section we will use the following notations. If $k$ is a nonnegative integer, then we will write $H_1(k)$ as shorthand for $H_{p,1,S}(|S|^{-(p-1)k}/2p)$. For $L:\F_p^n \to \F_p$ an affine polynomial, $I$ a subset of $[n]$ and $y$ an element of $S^I$, we will write \[L(y \to I): \F_p^{[n] \setminus I} \to \F_p\] for the affine polynomial obtained by evaluating $x_i$ as $y_i$ for each $i \in I$. We will write $\{y\} \times S^{I^c}$ for the subset \[\{x \in S^n: x_i = y_i \text{ for all } i \in I\}.\]

The first component of the proof of Proposition \ref{Degree-2 result} is the following lemma. It will play the role of an inductive step in the proof of Proposition \ref{Degree-2 result}, even though the proof of this lemma also contains an induction within.

\begin{lemma} \label{Inductive step in the d=2 case}

Let $p \ge 3$ be a prime, let $S$ be a non-empty subset of $\F_p$ containing at least two elements, and let $P$ be a degree-$2$ polynomial such that $P(S^n) \neq \F_p$. If $P$ has a $k$-square-$l$-determined decomposition for some nonnegative integers $k,l$ with $k \ge 2$, then there exists some $0 \le k' < k$ such that $P$ has a $k'$-square-$(l+(k-k')H_1(k))$-determined decomposition.

\end{lemma}

\begin{proof}

The assumption provides a square-determined decomposition \eqref{square-determined decomposition}. Let $I = I(J)$ be the set of coordinates on which the polynomial $J$ depends. Assume for contradiction that there exist two distinct indices in $[k]$, which without loss of generality we can take to be $k-1$ and $k$, which satisfy both inequalities \begin{align*} |Z(L_{k-1} - \sum_{i=1}^{k-2} a_i L_i) \cap I^c| & > H_1(k) \nonumber \\ |Z( L_{k} - \sum_{i=1}^{k-1} a_i L_i ) \cap I^c| & > H_1(k) \end{align*} respectively for all $a \in \F_p^{k-1}$ and for all $a \in \F_p^{k}$. Let us fix $y \in S^I$ arbitrarily. Then the affine polynomials $L_i(y \to I)$ with $i=1, \dots, k$ satisfy both inequalities \begin{align} |Z (L_{k-1}(y \to I) - \sum_{i=1}^{k-2} a_i L_i(y \to I) )| > H_1(k) \nonumber \\ |Z( L_{k}(y \to I) - \sum_{i=1}^{k-1} a_i L_i(y \to I) )| > H_1(k) \label{The kth affine polynomial is approximately spanned} \end{align} respectively for all $a \in \F_p^{k-1}$ and for all $a \in \F_p^{k}$, so by Proposition \ref{Conditional full range on restricted alphabets} and \eqref{square-determined decomposition} the image $P(\{y\} \times S^{I^c})$ contains some translate of $A_{k-1} Q_p + A_k Q_p$; since $|Q_p| \ge (p+1)/2$, the Cauchy-Davenport theorem shows that $P(\{y\} \times S^{I^c}) = \F_p$ and hence $P(S^n) = \F_p$, which contradicts our assumption. 

Therefore, we may without loss of generality assume that \eqref{The kth affine polynomial is approximately spanned} is false and hence that \[L_k = \sum_{i=1}^{k-1} a_i L_i + L_k^{\mathrm{rem}}\] for some $a \in \F_p^{k-2}$ and some affine polynomial $L_k^{\mathrm{rem}}$ with support contained in a union $I'$ of $I$ and of at most $H_1(k)$ elements. The set $I'$ hence has size at most $H_1(k)+|I|$. Substituting in \eqref{square-determined decomposition} and expanding, we obtain a decomposition of the type \[P = Y \circ (L_1, \dots, L_{k-1}) + U \circ (x_j, L_i: j \in I', i \in [k-1]) + W \circ (L_1, \dots, L_{k-1}) + J',\] where $Y: \F_p^{k-1} \to \F_p$ is a quadratic form, $U: \F_p^{I'} \otimes \F_p^{k-1} \to \F_p$ is a bilinear form, $W: \F_p^k \to \F_p$ is a linear form, and $J'$ is a polynomial with degree at most $2$ satisfying $I(J) \subset I'$. We can diagonalise the quadratic form $Y$ as \[Y = A_1' \Omega_1^2 + \dots + A_{k-1}' \Omega_{k-1}^2\] for some linearly independent linear forms $\Omega_1, \dots, \Omega_{k-1}: \F_p^{k-1} \to \F_p$ and some coefficients $A_1',\dots A_{k-1}' \in \F_p$. We obtain \begin{equation} Y \circ (L_1, \dots, L_{k-1}) = A_1' L_1'^2 + \dots + A_{k-1}' L_{k-1}'^2 \label{diagonalisation step} \end{equation} for some affine polynomials \[L_1' = \Omega_1 \circ (L_1, \dots, L_{k-1}), \dots, L_{k-1}' = \Omega_{k-1} \circ (L_1, \dots, L_{k-1})\] satisfying  \[\langle L_1', \dots, L_{k-1}' \rangle = \langle L_1, \dots, L_{k-1} \rangle, \] which allows us to furthermore rewrite \[U \circ (x_j, L_i: j \in I', i \in [k-1])\] as a linear combination of terms of the type $x_j L_i'$ with $j \in I'$ and $i \in [k-1]$, and \[W \circ (L_1, \dots, L_{k-1})\] as a linear combination of $L_1', \dots, L_{k-1}'$. The resulting decomposition is hence of the type \begin{equation} P = \sum_{i=1}^{k-1} \left( A_i' L_i'^2 + \sum_{j \in I'} U_{i,j} x_j L_i' + B_i' L_i' \right) + J' \label{expression before splitting into cases in the degree-2 proof} \end{equation} for some coefficients $U_{i,j}, B_i' \in \F_p$. If for some $i \in [k]$ we have $A_i' = 0$, $B_i' = 0$, and $U_{i,j} = 0$ for all $j \in I'$ then the expression \eqref{expression before splitting into cases in the degree-2 proof} may be rewritten with a smaller value of $k$, so we may assume without loss of generality that this does not happen for any $i \in [k]$. Let us now distinguish two cases.

\textbf{Case 1:} Assume that all coefficients $A_1', \dots, A_{k-1}'$ are non-zero. Then, letting \[L_i'' = L_i' + \sum_{j \in I'} (U_{i,j}/2A_i') x_j + B_i'/2A_i'\] for each $i \in [k-1]$ we may rewrite \[P = \sum_{i=1}^{k-1} A_i' L_i''^2 + J''\] for some polynomial $J''$ with degree at most $2$ satisfying $I(J'') \subset I'$. Compared to \eqref{square-determined decomposition}, the number of terms in the summation has decreased by $1$, at the cost of $J''$ depending on at most $H_1(k)$ more coordinates than $J$ does. This finishes the proof in this case.

\textbf{Case 2:} Assume that one of the coefficients $A_1', \dots, A_{k-1}'$ is zero. Without loss of generality we may assume it to be $A_{k-1}'$. Then by the assumption that we have made in the previous paragraph, the terms $B_i'$ and $U_{i,j}$ with $j \in I'$ cannot all be zero, so because $S$ has size $2$ or more, there necessarily exists $y \in S^{I'}$ such that \begin{equation} \sum_{j \in I'} U_{i,j} y_j + B_i' \neq 0. \label{non-zero linear term} \end{equation} If we have \[|Z( L_{k-1}' - \sum_{i=1}^{k-2} a_i L_i') \cap I^c| > H_1(k)\] for every $a \in \F_p^{k-2}$ then the affine polynomials $L_i'(y \to I')$ with $i=1, \dots, k-1$ obtained by evaluating the coordinate $x_j$ to $y_j$ for each $j \in I'$ satisfy \[|Z ( L_{k-1}'(y \to I') - \sum_{i=1}^{k-2} a_i L_i'(y \to I') )| > H_1(k) \] for all $a \in \F_p^{k-2}$. Using Proposition \ref{Conditional full range on restricted alphabets} together with \eqref{expression before splitting into cases in the degree-2 proof} and \eqref{non-zero linear term} shows $P(\{y\} \times S^{I^c}) = \F_p$ and in particular $P(S^n) = \F_p$, contradicting our assumption. Therefore, we may write \[ L_{k-1}' = \sum_{i=1}^{k-2} a_i L_i' + L_{k-1}'^{\mathrm{rem}}\] for some affine polynomial $L_{k-1}'^{\mathrm{rem}}$ with support contained in a union $I''$ of $I'$ and of at most $H_1(k)$ elements.  We may then rewrite \eqref{expression before splitting into cases in the degree-2 proof} as \[P = \sum_{i=1}^{k-2} \left( A_i' L_i'^2 + \sum_{j \in I'} U_{i,j} x_j L_i' + B_i' L_i' \right) + J''\] for some polynomial $J''$ with degree at most $2$ satisfying $I(J'') \subset I''$. Compared to \eqref{expression before splitting into cases in the degree-2 proof}, the number of terms in the summation has decreased by $1$, at the cost of $J''$ depending on at most $H_1(k)$ more coordinates than $J$ does. We again split into cases, and iterate until we are no longer in Case 2. \end{proof}

The other main component of the proof of Proposition \ref{Degree-2 result} is the special case $d=2$ of the equidistribution statement, Theorem \ref{Equidistribution on restricted alphabets}, which will give us that $P$ coincides on $S^n$ with a $k$-square-$l$-determined decomposition for some bounded integers $k,l$, at which we may start the inductive process of which Lemma \ref{Inductive step in the d=2 case} encapsulates an iteration. We note for the reader interested in an utterly self-contained proof of Proposition \ref{Degree-2 result} that the proof of this special case of Theorem \ref{Equidistribution on restricted alphabets} involves far fewer difficulties than the proof of the theorem in full generality: the proof of that special case is carried out in Section 2 of \cite{Gowers and K equidistribution}.

\begin{proof}[Proof of Proposition \ref{Degree-2 result}]

The result is immediate if $S$ has size $1$, so we may assume $|S| \ge 2$. The result is also immediate if $p=2$, since if $P(S^n) \neq \F_2$ then $P$ coincides with a constant on $S^n$, so we may also assume $p \ge 3$. Because $P(S^n)$ is not the whole of $\F_p$, applying Theorem \ref{Equidistribution on restricted alphabets} shows that $P$ coincides on $S^n$ with some polynomial (which we will continue referring to as $P$ throughout the remainder of the proof) that we may decompose as \[ P = \Lambda_{1,1} \Lambda_{1,2} + \dots + \Lambda_{r,1} \Lambda_{r,2}\] for some $r \le H_{p,2,S}(1/2p)$ and some affine polynomials $\Lambda_{1,1}, \Lambda_{1,2} \dots, \Lambda_{r,1}, \Lambda_{r,2}$ (which may be constant polynomials). Writing $L_{i,1}$ (resp. $L_{i,2})$ for the linear part of each $\Lambda_{i,1}$ (resp. $\Lambda_{i,2}$) we obtain a decomposition \begin{equation} P = Q + L_0 \label{Decomposition of P into degree parts} \end{equation} where \[Q = L_{1,1} L_{1,2} + \dots + L_{r,1} L_{r,2}\] is the quadratic part of $P$ and $L_0$ is an affine polynomial. We may rewrite \[ Q = A_1 L_1^2 + \dots + A_k L_k^2 \] for some $k \le 2r$, some non-zero coefficients $A_1, \dots, A_k \in \F_p$, and some linear polynomials $L_1, \dots, L_k$. If \[|Z( L_0 - \sum_{i=1}^{k} b_i L_i)| > H_1(k) \] for all $b \in \F_p^k$, then by choosing $(v_1, \dots, v_k)$ such that \[\P_{x \in S^n}(L_1(x)=v_1, \dots, L_k(x)=v_k) > 0\] and applying Proposition \ref{Conditional full range on restricted alphabets} we obtain that $(L_0, L_1, \dots, L_k)(S^n)$ contains $\F_p \times \{(v_1, \dots, v_k)\}$, and the decomposition \eqref{Decomposition of P into degree parts} then shows that the image $P(S^n)$ contains the whole of $\F_p$, which contradicts our assumption. Therefore, we may write \[ L_0 = \sum_{i=1}^{k} B_i L_i + L_0^{\mathrm{rem}} \] for some $B \in \F_p^k$ and some affine polynomial $L_0^{\mathrm{rem}}$ with support size at most $H_1(k)$. In turn, this leads to the decomposition \[ P = \sum_{i=1}^k (A_i L_i^2 + B_i L_i) + L_0^{\mathrm{rem}}. \] Because $A_1, \dots, A_k$ are all non-zero, we may assume without loss of generality that $B_1, \dots, B_k$ are all zero, since it suffices to replace $L_i$ by $L_i + B_i/2A_i$ for each $i \in [k]$: doing so only adds a constant to $L_0^{\mathrm{rem}}$.

We have now obtained a $k$-square-$H_1(k)$-determined decomposition \eqref{square-determined decomposition} of $P$. Iterating Lemma \ref{Inductive step in the d=2 case} we eventually get a $1$-square-$kH_1(k)$-determined decomposition or a $0$-square-$(k+1)H_1(k)$-determined decomposition. This establishes item 1, aside from the value of the bound. To finish the proof of item 2, we furthermore add one last step. In the second situation we are done, and in the first situation we can write \[ P = A L_{\mathrm{one}}^2 + J_{\mathrm{one}}\] for some $A \in \F_p$, some affine polynomial $L_{\mathrm{one}}$ and some degree-$2$ polynomial $J_{\mathrm{one}}$ such that the set $I_{\mathrm{one}} = I(J_{\mathrm{one}})$ of coordinates on which $J_{\mathrm{one}}$ depends has size at most $kH_1(k)$. If $A=0$ then we are done, so let us assume that this is not the case. If the support of $L$ contains at least $H_1(k)+1$ elements outside $I_{\mathrm{one}}$, then for any (arbitrary) $y \in S^{I_{\mathrm{one}}}$ we have \[|Z (L(y \to I_{\mathrm{one}}))| > H_1(k),\] so by Proposition \ref{Conditional full range on restricted alphabets} the image $P(S^n)$ contains a translate of $A Q_p$, which contradicts the assumption of item 2. We conclude that $P$ is $(k+1)H_1(k)$-determined.

All that remains to do is to compute the bounds. By \cite{Gowers and K equidistribution}, Proposition 2.11 and then \cite{Gowers and K equidistribution}, Proposition 2.2 respectively there exist some absolute constants $\k''$,$\k'$ such that $k \le \k''p^5$ and $H_1(k) \le \k'p^{10}$, so $(k+1)H_1(k) \le \kappa p^{15}$ for some absolute constant $\k$, as desired. \end{proof}

\section{Deduction of the main theorem in the general case}\label{Section: Deduction of the main theorem in the general case}

In this section we prove Theorem \ref{Main result on ranges}. There too, we can conclude immediately if $|S|=1$, so throughout the remainder of the proof we will assume $|S| \ge 2$. We begin by defining the kinds of decompositions that the proof will focus on, and the hierarchy that we will use between them.

\begin{definition}

Let $p$ be a prime, let $1 \le d < p$ be an integer, and let $S$ be a non-empty subset of $\F_p$. We say that an \emph{acceptable decomposition} (with respect to $S$) of a polynomial $P:\F_p^n \to \F_p$ with degree $d$ is a pair $(\cf, F)$ with $\cf$ a finite set of non-zero polynomials each with degree at most $d$, and \[F: \{(i,j): i \in [r], j \in J_i\} \to \cf\] a function for some nonnegative integer $r$ and some finite sets $J_i$ for each $i \in [r]$, which satisfies the following two properties. \begin{enumerate} \item We have the decomposition
\begin{equation} P = P_0 + \sum_{i=1}^r \a_i \prod_{j \in J_i} P_{F(i,j)} \label{Definition of acceptable decomposition} \end{equation} for some polynomial $P_0$ with degree at most $d$ such that $P_0(S^n) = \{0\}$ and for some coefficients $\a_i \in \F_p$. \item Every $i \in [r]$ satisfies \[\deg \prod_{j \in J_i} P_{F(i,j)} \le d.\] \end{enumerate}

\end{definition}

\begin{definition}
Let $p$ be a prime, let $1 \le d < p$ be an integer, and let $S$ be a non-empty subset of $\F_p$. Let $P:\F_p^n \to \F_p$ be a polynomial with degree $d$. We say that the \emph{degree description} of a given acceptable decomposition $(\cf, F)$ of $P$ is the $(d+1)$-tuple $(D_0,\dots,D_d) \in \N^d$ such that 

\begin{enumerate}[]

\item $D_0$ is the number of polynomials of $\mathcal{F}$ which are either of the type $ax_i$ for some $i \in [n]$ and $a \in \F_p^*$, or of the type $c$ for some $c \in \F_p$
\item $D_1$ is the number of all other affine polynomials of $\mathcal{F}$
\item $D_u$ is the number of degree-$u$ polynomials of $\mathcal{F}$ for every $u \in \{2,\dots,d\}$.

\end{enumerate}

We say that the \emph{size} of the acceptable decomposition is the size of the set $\cf$. For every polynomial $Q \in \cf$, we furthermore say that the \emph{modified degree} of the polynomial $Q$ is the unique $u$ of $\{0,1,\dots, d\}$ such that $Q$ is of the type described above in the definition of $D_u$.

\end{definition}

Let $P$ be as in the assumption of Theorem \ref{Main result on ranges}, and let $e = \lfloor d/(t+1) \rfloor$. The main idea of the proof will be to successively obtain a sequence of acceptable decompositions of $P$ which have degree descriptions that are strictly decreasing in the colexicographic order, until we obtain an acceptable decomposition in which all polynomials of $\cf$ have modified degree at most $e$. For each step $l$ of the process we write $\cf_l$ for the family $\cf$ at the start of the $l$th acceptable decomposition of the process. The three following properties will be satisfied.

\begin{enumerate}[(i)]

\item The difference $\mathcal{F}_{l} \setminus \mathcal{F}_{l+1}$ is a singleton containing a polynomial of $\mathcal{F}_{l}$ with maximal modified degree among those of the polynomials of $\cf_l$.
\item The difference $\mathcal{F}_{l+1} \setminus \mathcal{F}_{l}$ has size bounded by a function of the degree description of $\mathcal{F}_{l}$ (and of $p,d,t,S$).
\item All polynomials of $\mathcal{F}_{l+1} \setminus \mathcal{F}_{l}$ have modified degrees strictly smaller than that of the polynomial of $\mathcal{F}_{l} \setminus \mathcal{F}_{l+1}$.

\end{enumerate}

The acceptable decomposition obtained at the end of the process will still have bounded size, with bounds depending on $p,d,t,S$ only. For the convenience of the reader we will formalise the main principle behind this as a separate lemma. Although we have found it more convenient to state it in its somewhat abstract form, the case where we will use it is that where $K(D)$ is the size of a given acceptable decomposition obtained at the end of an induction starting with an acceptable decomposition with degree description $D$, the quantity $V(D)$ is the size $D_0 + \dots + D_e$ whenever $D_{e+1} = \dots = D_d = 0$ (since if the original acceptable decomposition has all its polynomials with degrees at most $e$, then there is nothing to do), and $W(D)$ is a bound on the number of extra polynomials coming from item (ii), when the family $\cf_l$ has degree description $D$.

\begin{lemma}\label{Lemma on bounds in the induction}

Let $0 \le e \le d$ be nonnegative integers. Let $V,W: \N^{d+1} \to \N$ be functions. Then there exists a function $B: \N^{d+1} \to \N$ depending on $e,d,V,W$ only such that whenever $K: \N^{d+1} \to \N$ is a function satisfying

\[K(D_0, D_1, \dots, D_e, 0, \dots, 0) \le V(D_0, D_1, \dots, D_e, 0, \dots, 0)\] for all $D_0, D_1, \dots, D_e \in \N$ and such that for each integer $m$ with $e < m \le d$ and all $D_0, D_1, \dots, D_{m} \in \N$ with $D_m \neq 0$ there exist \[0 \le u_0, \dots, u_{m-1} \le W(D_0, D_1, \dots, D_{m}, 0, \dots, 0)\] satisfying \[K(D_0, D_1, \dots, D_{m}, 0, \dots, 0) \le K(D_0+u_0, D_1+u_1, \dots, D_{m-1}+u_{m-1}, D_{m}-1, 0, \dots, 0)\] then $K(D) \le B(D)$ for all $D \in \N^{d+1}$.

\end{lemma}

\begin{proof}

We proceed by induction. The first assumption shows that we may take $B(D)=V(D)$ for all $D \in \N^{d+1}$ such that $D_m = 0$ for all $e<m\le d$, which finishes the base case. Then, assuming for contradiction that there exists some other $D = (D_1, \dots, D_m, 0, \dots 0)$ with $D_m \neq 0$ for which the conclusion is not satisfied, we consider the smallest such $D$ for the colexicographic order on $\N^{d+1}$; but then, the second assumption shows that taking \[B(D) = \max_{0 \le u_0, \dots, u_{m-1} \le W(D_0, D_1, \dots, D_{m}, 0, \dots, 0)} B(D_0+u_0, D_1+u_1, \dots, D_{m-1}+u_{m-1}, D_{m}-1, 0, \dots, 0)\] does ensure $K(D) \le B(D)$, a contradiction. \end{proof}

There remains the matter of how the successive acceptable decompositions will be obtained.

\begin{proof}[Proof of Theorem \ref{Main result on ranges}]

Let $P$ be as in the assumption of Theorem \ref{Main result on ranges}, that is, let us assume that $P(S^n)$ does not contain $A(\F_p)$ for any non-constant polynomial $A:\mathbb{F}_p \to \mathbb{F}_p$ with degree at most $t$. Our starting acceptable decomposition consists in writing $P$ merely as itself: this provides a decomposition of size $1$ and with degree description $(0,\dots,0,1)$.

The image $P(S^n)$ is in particular not the whole of $\F_p$, so by Theorem \ref{Equidistribution on restricted alphabets}, the polynomial $P$ coincides on $S^n$ with some polynomial of the type \[P = Q_1 R_1 + \dots + Q_{s} R_s\] for some $s \le H_{p,d,S}(1/2p)$ and some polynomials $Q_1, \dots, Q_s, R_1, \dots, R_s$ satisfying \[\deg Q_i \le d-1, \deg R_i \le d-1 \text{ and }\deg Q_i + \deg R_i \le d\] for each $i \in [s]$. This is in particular an acceptable decomposition $(\cf_1, F_1)$ with size at most $2H_{p,d,S}(1/2p)$ and with degree description $(D_0,\dots,D_{d-1},0)$ for some \[D_0, \dots, D_{d-1} \le 2H_{p,d,S}(1/2p).\]

After this, diminishing the degree description becomes more complicated, and lends itself to the inductive step that we will now describe. Assume that for some $l \ge 1$ we have obtained an acceptable decomposition $(\cf_l, F_l)$ of $P$, and let us construct $(\cf_{l+1}, F_{l+1})$. Let $s$ be the size of the decomposition $(\cf_l, F_l)$. Recall that $e = \lfloor d/(t+1) \rfloor$. We now distinguish three cases. Letting $k$ be the size of $\cf_l$, we may write $\cf_l= \{P_1, \dots, P_k\}$ for some polynomials $P_1, \dots, P_k$.

\textbf{Case 1:} All polynomials $P_1, \dots, P_k$ have modified degrees at most $e$. Then we stop the induction.

We now assume that we are not in Case 1. Without loss of generality we may assume that $P_k$ has modified degree strictly greater than $e$, and has maximal modified degree $m$ among the modified degrees of $P_1, \dots, P_k$. By gathering terms inside \eqref{Definition of acceptable decomposition} depending on the power of $P_{k}$ appearing inside the term, we may hence rewrite the decomposition \eqref{Definition of acceptable decomposition} as \begin{equation} P = \sum_{r=0}^t (T_r \circ (P_1, \dots, P_{k-1})) P_{k}^r \label{decomposition with respect to the last polynomial} \end{equation} for some polynomials $T_0, \dots, T_{t}: \F_p^{k-1} \to \F_p$, such that the degree of each summand is still at most $d$: in particular this is why we may stop the summation at $t$.

\textbf{Case 2:} The image \[(T_1 \circ (P_1, \dots, P_{k-1}), \dots, T_t \circ (P_1, \dots, P_{k-1}))(S^n)\] is equal to $\{0\}^t$. Then $P$ coincides on $S^n$ with the polynomial $T_0 \circ (P_1, \dots, P_{k-1})$. This provides a new acceptable decomposition $(\cf_{l+1}, F_{l+1})$ of $P$, with $\cf_{l+1} = \cf_l \setminus \{P_k\}$. We can check that the families $\cf_l$ and $\cf_{l+1}$ satisfy items (i)-(iii). We now assume that we are not in Case 2 either.

\textbf{Case 3:} Assume for contradiction that the polynomial $P_{s}$ satisfies \begin{equation} \rk_{m - 1,S} \left( P_{k} - \sum_{i=1}^{k-1} a_i P_i \right) > H_{p,m,S}(|S|^{-(p-1)mk} / 2p) \label{rank inequality in the proof of the main theorem} \end{equation} for all $a \in \F_p^{k-1}$. Because we are not in Case 2, the image \[(T_0 \circ (P_1, \dots, P_{k-1}), T_1 \circ (P_1, \dots, P_{k-1}), \dots, T_t \circ (P_1, \dots, P_{k-1}))(S^n)\] contains some element $h$ of $\F_p^{t+1}$ which has its $t$ last coordinates not all equal to $0$, so there exists some $y \in \F_p^{k-1}$ in the image $(P_1, \dots, P_{k-1})(S^n)$ such that $(T_1, \dots, T_{t})(y) \neq 0$. By \eqref{rank inequality in the proof of the main theorem} and Proposition \ref{Conditional full range on restricted alphabets}, the image $(P_k,P_1, \dots, P_{k-1})(S^n)$ contains $\F_p \times \{y\}$, so the decomposition \eqref{decomposition with respect to the last polynomial} shows that the image $P(S^n)$ contains the image $A(\F_p)$, where $A:\F_p \to \F_p$ is the polynomial defined by \[A(u) = h_0 + h_1 u + \dots + h_t u^t.\] By construction the polynomial $A$ is non-constant and has degree at most $t$, which contradicts the assumption of Theorem \ref{Main result on ranges}. We can hence obtain some $a \in \F_p^{k-1}$, some \[k' \le H_{p,m,S}(|S|^{-(p-1)mk} / 2p),\] some polynomials $P_1', \dots, P_{k'}'$ each with modified degree strictly less than the modified degree $m$ of $P_{k}$ and some polynomial $P_0'$ with degree at most $m$ satisfying $P_0'(S^n) = \{0\}$, such that we can write \begin{equation} P_k = P_0' + \sum_{i=1}^{k-1} a_i P_i + \sum_{i=1}^{k'} P_i'. \label{linear combination of the new polynomial} \end{equation} Replacing each instance of $P_{k}$ by the linear combination \eqref{linear combination of the new polynomial} in the expression \eqref{Definition of acceptable decomposition} (and gathering the contribution of $P_0'$ together with $P_0$ there) leads to a new acceptable decomposition $(\cf_{l+1}, F_{l+1})$ of $P$, where \[\cf_{l+1} = (\cf_l \setminus \{P_k\}) \cup \{P_1', \dots, P_{k'}'\}.\] Again, the families $\cf_l$ and $\cf_{l+1}$ satisfy items (i)-(iii). This concludes the inductive step.

Items (i)-(iii) and Lemma \ref{Lemma on bounds in the induction} then show that there exists some integer $\Psi(p,d,t)$ such that we can obtain an acceptable decomposition $(\cf, F)$ of $P$ with size at most $\Psi(p,d,t)$ and such that all its polynomials in $\cf$ have modified degree at most $e$. We define the quantities \[C_{\mathrm{pre}}(p,d,t) = \sum_{0 \le v \le d} \Psi(p,d,t)^v \text{ and } C(p,d,t) = p^{C_{\mathrm{pre}}(p,d,t)}.\]The number of products of elements of $\cf$ which are polynomials with degree at most $d$ is at most $C_{\mathrm{pre}}(p,d,t)$, so the number of linear combinations of them is at most $C(p,d,t)$ and we conclude as desired that \[\rk_{e,S} P \le C(p,d,t). \qedhere \] \end{proof}

\section{Further questions}\label{Section: Open problems}

Our results still leave open several questions regarding what we may infer about the structure of a polynomial from its range.

Let us begin by looking at degree-$2$ polynomials. As we discussed at the very end of the proof of Proposition \ref{Degree-2 result}, the special case of Theorem \ref{Equidistribution on restricted alphabets} for degree-$2$ polynomials, Proposition 2.11 of \cite{Gowers and K equidistribution}, implies in particular that there exists an absolute constant $\k''>0$ such that if $P$ is a degree-$2$ polynomial satisfying $P(S^n) \neq \F_p$, then $\rk_{1,S} P \le \k'' p^5$. We may ask whether this power bound in $p$ may be improved to a linear bound, and even what the optimal bound is.

In the case where $P$ has degree $1$ and $|S| \ge 2$, the Cauchy-Davenport theorem provides an analogous bound $\rk_0 P \le \frac{p-2}{|S|-1}$ on the size of any such polynomial (if $|S|=1$, then $\rk_0 P \le 1$ automatically). The value $\frac{p-2}{|S|-1}$ is attained for instance for $S = \{0,1\}$ and \[P = x_1 + \dots + x_{p-2}.\] A degree-$2$ analogue of this example is the polynomial \[P = x_1x_2 + \dots + x_{2p-5} x_{2p-4},\] which has degree-$1$ rank equal to $p-2$. Indeed $P$ has rank $2p-4$ as a quadratic form, and it is straightforward to check that the rank of a quadratic form is at most twice its degree-$1$ rank. We have not found an example that would settle the following question negatively, and although this question would vaguely appear to be a kind that could potentially be approached using the combinatorial Nullstellensatz, we have not managed to use this technique to settle it positively either.

\begin{question} Let $p$ be a prime, and let $S$ be a non-empty subset of $\F_p$. Let P$:\F_p^n \to \F_p$ be a degree-$2$ polynomial. Is it true that if $P(S^n) \neq \F_p$ then $\rk_{1,S} P \le p-2$ ? \end{question}

More generally, we can ask for better bounds in Theorem \ref{Main result on ranges}. Even if we assume that the bounds on $H_{p,d,S}(\epsilon)$ from Theorem \ref{Equidistribution on restricted alphabets} can be taken, for a fixed $p$, to be linear in $\log \epsilon^{-1}$ (something which, at the time of writing, is not yet known to be true even for $S = \F_p$), the degree-$1$ case (\cite{Gowers and K equidistribution}, Proposition 2.2) already has an inverse square dependence on $p$. So for any fixed degree $3 \le d < p$, the number of iterations on the degree description in the proof of Theorem \ref{Main result on ranges} would still have a worse than exponential growth rate in $p$, and the final bounds in Theorem \ref{Main result on ranges} would hence be correspondingly weak.

\begin{question} Can we obtain polynomial bounds in $p$ in Theorem \ref{Main result on ranges} ? \end{question}

In a somewhat similar direction, we may ask whether there is some other proof of Theorem \ref{Main result on ranges}. That the proof of Theorem \ref{Main result on ranges} in the present paper is rather short is due to the fact that it uses Theorem \ref{Equidistribution on restricted alphabets} as a black box: as discussed in \cite{Gowers and K equidistribution}, the proof of Theorem \ref{Equidistribution on restricted alphabets} itself uses as black boxes the facts that high partition implies high analytic rank and that high essential partition rank implies high essential disjoint partition rank, with the latter implication in the sense of \cite{Gowers and K equidistribution}, Definition 1.9 and \cite{Gowers and K equidistribution}, Theorem 1.10, which in turn was proved in \cite{K} as a special case of \cite{K}, Theorem 1.8.

\begin{question} Does Theorem \ref{Main result on ranges} have a purely combinatorial proof ? \end{question}

As we explained in the introduction, we cannot hope, in the statement of Theorem \ref{Main result on ranges}, to obtain that if a polynomial $P: \F_p^n \to \F_p$ satisfies $P(S^n) \neq \F_p$, then the polynomial $P$ coincides on $S^n$ with the sum of a \emph{single} polynomial of the type $A \circ Q$ and of a polynomial determined by a bounded number of coordinates. Nonetheless, this does not mean that obtaining a variant of that description is necessarily impossible. For fixed $1 \le d < p$, let $m(p,d)$ be the smallest positive integer such that whenever $A_1, \dots A_{m}: \F_p \to \F_p$ are non-constant polynomials each with degree at most $d$ we have \[A_1(\F_p) + \dots + A_m(\F_p) = \F_p.\] We can correspondingly formulate the most optimistic refinement of Theorem \ref{Main result on ranges} that we may still hope for.

\begin{question}

Let $p$ be a prime, let $1 \le d < p$ be an integer and let $S$ be a non-empty subset of $\F_p$. Is it true that if $P: \F_p^n \to \F_p$ is a polynomial with degree $d$ satisfying $P(S^n) \neq \F_p$, then $P$ coincides on $S^n$ with some polynomial of the type \[A_1 \circ Q_1 + \dots + A_l \circ Q_l + J\] for some integer $l < m$, some polynomials $Q_1, \dots, Q_l$ with degrees at most $\lfloor d/2 \rfloor$, some polynomials $A_1, \dots, A_l:\F_p \to \F_p$ such that the polynomials $A_1 \circ Q_1, \dots,  A_l \circ Q_l$ all have degrees at most $d$, and some $k$-determined polynomial $J$ with degree at most $d$, for some $k$ that is bounded depending on $p,d,t$ (and $S$) only ?

\end{question}

One possibility would be to try to imitate the proof of Proposition \ref{Degree-2 result}, but it breaks down in a key way in that the obvious would-be analogue of equation \eqref{diagonalisation step} for higher degrees does not hold: unlike for $d=2$ it is not true for $d \ge 3$ that if $k$ is a positive integer, then a homogeneous polynomial of degree $d$ in $k$ variables can be written as a sum of at most $k$ $d$-powers.

We note that using the polarisation identity on multilinear polynomials shows that a polynomial with degree at most $1\le d<p$ and of degree-$\lfloor d/(t+1) \rfloor$ rank at most some positive integer $r$ can always be rewritten as a linear combination of at most $2^dr$ powers of polynomials with degrees at most $d\lfloor d/(t+1) \rfloor$. This may be viewed as going some way in establishing a result of this type, although, as is shown by the example \[4 (x_1^3x_2 + x_3^2x_4^2) = (x_1^3+x_2)^2 - (x_1^3-x_2)^2 + (x_3^2+x_4^2)^2 - (x_3^2-x_4^2)^2,\] the powers may have degrees larger than the degree of the original polynomial, and this is why we have refrained from defining our notions of rank in terms of powers rather than products in Definition \ref{rank of a polynomial} and in Definition \ref{further definitions of rank of a polynomial}.

Let us finish by a question which may have occurred to the reader, although it is purely speculative.

\begin{question} Does the case $S = \{0,1\}$ of Theorem \ref{Main result on ranges} have any computer science applications ? \end{question}


\begin{thebibliography}{9}


\bibitem{Adiprasito Kazhdan Ziegler}

K. Adiprasito, D. Kazhdan, and T. Ziegler, \textit{On the Schmidt and analytic ranks for trilinear forms}, arXiv:2102.03659 (2021).

\bibitem{Bhowmick and Lovett}

A. Bhowmick and S. Lovett, \textit{Bias vs structure of polynomials in large fields, and applications in effective algebraic geometry and coding theory}, IEEE Trans. Inf. Theory, arXiv:1506.02047 (2015).

\bibitem{Green and Tao}

B. Green and T. Tao, \textit{The distribution of polynomials over finite fields, with applications to the Gowers norms.} Contr. Discr. Math., \textbf{4} (2009), no. 2, 1-36.

\bibitem{Gowers and K equidistribution}

W. T. Gowers and T. Karam, \textit{Equidistribution of high-rank polynomials with variables restricted to subsets of $\F_p$}, arXiv:2209.04932 (2022).

\bibitem{Janzer}

O. Janzer, \textit{Polynomial bound for the partition rank vs the analytic rank of tensors}, Discrete Anal. \textbf{7} (2020), 1-18.

\bibitem{K}

T. Karam, \textit{High-rank subtensors of high-rank tensors}, arXiv:2207.08030 (2022).

\bibitem{Kaufman and Lovett}

T. Kaufman and S. Lovett, \textit{Worst case to average case reductions for polynomials}, 49th Annual IEEE Symposium on Foundations of Computer Science (2008), 166-175.

\bibitem{Milicevic}

L. Mili\'cevi\'c, \textit{Polynomial bound for partition rank in terms of analytic rank}, Geom. Funct. Anal. \textbf{29} (2019), 1503-1530.

\bibitem{Moshkovitz and Cohen 1}

A. Cohen and G. Moshkovitz, Structure vs. randomness for bilinear maps, Discrete Anal. \textbf{12} (2022).

\bibitem{Moshkovitz and Cohen 2}

A. Cohen and G. Moshkovitz, Partition and analytic rank are equivalent over large fields, (2022).

\bibitem{Moshkovitz and Zhu}

G. Moshkovitz and D. G. Zhu, \textit{Quasi-linear relation between partition and analytic rank}, arXiv:2211.05780 (2022).





\end{thebibliography}
\end{document}